\documentclass[11pt,a4paper]{article}

\usepackage{amssymb}
\usepackage{amsmath}
\usepackage{amsthm}
\usepackage{amsfonts}
\usepackage{epsfig}
\usepackage{marvosym}

\usepackage[english]{babel}
\usepackage{color}

\newtheorem{defi}{Definition}
\newtheorem{prop}[defi]{Proposition}
\newtheorem{cor}[defi]{Corollary}
\newtheorem{lem}[defi]{Lemma}
\newtheorem{theo}[defi]{Theorem}
\newtheorem{rem}[defi]{Remark}

\newcommand{\dx}{\, {\mathrm d}}

\newcommand\dt{{\frac{\mathrm d}{\mathrm dt}}}

\newcommand{\dd}{{\, \mathrm d}}

\newcommand{\E}{\mathbb E}
\newcommand{\N}{\mathbb N}
\newcommand{\R}{\mathbb R}
\newcommand{\PP}{\mathbb P}

\renewcommand{\epsilon}{\varepsilon}
\renewcommand{\phi}{\varphi}

\newcommand{\var}{\varepsilon}

\date{}
\author{
 Nicolas Meunier\footnote{Universit\'e Paris Descartes, MAP5, 45 rue des saints p\`eres 75006 Paris, France, \texttt{nicolas.meunier@parisdescartes.fr}, },
 Cl\'ement Mouhot\footnote{University of Cambridge, DPMMS, Centre for Mathematical Sciences
      Wilberforce road, Cambridge CB3 0WA, UK, \texttt{C.Mouhot@dpmms.cam.ac.uk}, },
 Rapha\"el Roux\footnote{Universit\'e Pierre et Marie Curie, LPMA,
 4 Place Jussieu 
75005 Paris, France, \texttt{raphael.roux@upmc.fr}}
}

\title{Long time behavior in locally activated random walks}

\begin{document}

\maketitle

\begin{abstract}
We consider a $1$-dimensional Brownian motion
whose diffusion coefficient varies when it crosses the origin. We study the long time behavior and we establish different regimes, depending on the variations of the diffusion coefficient: emergence of a non-Gaussian multipeaked
probability distribution and a dynamical transition to an absorbing
static state. We compute the generator and we study the partial differential equation which involves its adjoint. We discuss global existence and blow-up of the solution to this latter equation. 
\end{abstract}

\section{Introduction}

In this paper we deal with a new class of one dimensional linear diffusion problem in which the diffusivity is modified in a 
prescribed way upon each crossing of the origin. We study both the system of stochastic differential equations satisfied by the position and the diffusion coefficient of a brownian particle  whose diffusion coefficient is modified at each crossing of the origin and the partial differential equation satisfied by the joint distribution of the solution to the stochastic system. In both viewpoints we obtain non-trivial behaviors of the solution: dynamical transition to an absorbing state for the solution to the stochastic system and blow-up of the density of the joint distribution. Global existence versus blow-up has been widely studied for non-linear equations, such as for the Keller-Segel system in two dimensions of space see e.g. \cite{BDP}. In our case, the partial differential equation is linear and the instability driving the system towards an inhomogeneous state is the diffusion.

Living matter provides a prototypical example of such a problem: the dynamics of a cell or a bacterium in the presence of a
localized patch of nutrients, which enhances its ability to move, as for exemple the dynamics of a macrophage that grows by accumulating smaller and spatially localized particles, such as lipids, Figure \ref{macrophage} and \cite{athero}, or,
alternatively, a localized patch of toxins that impairs its mobility. In \cite{benichou-meunier-redner-voituriez-12}, to describe the movement
of such a particle, the following formal system of stochastic differential
equations was introduced:
\begin{equation}\label{eq:PRE}
\begin{cases}
\dx X_t&=\sqrt{2A_t}\dx W_t,\\
\dx A_t&=f(A_t)\delta _{X_t=0} \dx t,
\end{cases}
\end{equation}
where~$(W_t)_{t\geq0}$ is a given standard one-dimensional Brownian
motion,~$X_t$ and~$A_t$ respectively denote the position and
the diffusion
coefficient of the particle at time~$t$. 

\begin{figure} 
\begin{center}
\includegraphics[scale=0.5]{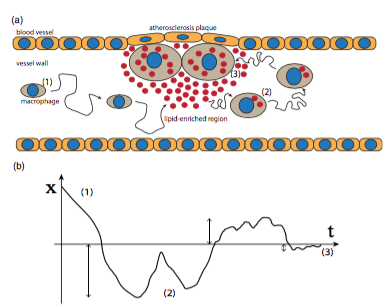}
\end{center}
\caption{a) Sketch of the different stages of atherosclerosis plaque formation: (1) rapid diffusion of a “free” macrophage cell; (2) upon entering a localized lipid-enriched region, the macrophage accumulates lipids, and thereby grows and becomes less mobile; and (3) after many crossings of the lipid-enriched region, the macrophage eventually gets trapped, resulting in the formation of an atherosclerotic plaque. (b) Sketch of a one-dimensional particle trajectory of the model of locally decelerated random walk.}
\label{macrophage}
\end{figure}

In the previous system, the term~$\delta_{X_t=0}$ does not
make sense. In \cite{benichou-meunier-redner-voituriez-12}, it was
understood in the sense that the generator of the Markov
process~$(X_t,A_t)_{t\geq0}$ was 
\begin{equation}\label{eq:generateur_PRE}
\mathcal Lh(x,a)
=a\partial_{xx}^2h(x,a)
+f(a)\partial_ah(x,a) \delta_{x=0},
\end{equation}
for any smooth function~$h$. Here, we are first interested to give a correct formulation of the term~$\delta_{X_t=0}$.
Intuitively, the term~$\delta_{X_t=0}\dx t$ should represent a measure on~$[0,\infty)$
giving full measure to the set of zeros of the
process~$(X_t)_{t\geq0}$. This reminds the notion of local time.

We point out that in the formal system~\eqref{eq:PRE}  at
any time, the diffusion coefficient, $A_t$, depends on the entire history of the trajectory. Thus the
evolution of the particle position, $X_t$, is intrinsically
non-Markovian. Despite these considerations,  in the particular case where $f$ is a power function, $f(a)=\pm a^\gamma$, $\gamma \ge 0$, we study the long-time behavior of the
process~$(X_t)_{t\geq0}$ solution to the system with the correct formulation of the term~$\delta_{X_t=0}$. 
Our main findings are: (i) The probability distribution
of the position has a non-Gaussian tail. (ii) For local acceleration, i.e. $f$ takes nonnegative values, $f(a)=a^\gamma$, a
diffusing particle is repelled from the origin, so that the maximum in
the probability distribution is at nonzero displacement. (iii) For local
deceleration, i.e. $f$ takes negative values, $f(a)=-a^\gamma$, a dynamical transition to an absorbing state occurs: for
sufficiently strong deceleration, $\gamma \in (0,3/2)$, the particle can get trapped at the
origin in finite time while if the deceleration process is
sufficiently weak, $\gamma \ge 3/2$, the particle never gets
trapped.

In a second step, we study the generator of the Markov
process~$(X_t,A_t)_{t\geq0}$ solution to the system with the correct formulation of the term~$\delta_{X_t=0}$. In order to do so we first prove that the generator of the Markov
process~$(W_t,L_t^W)_{t\geq0}$, in a weak sense, is given by
\[
\mathcal L_0h(w,l)=\frac12\partial_{ww}^2 h(w,l)+\partial_l h(w,l) \delta_{l=0},
\]
where $L_t^W$ is the local time at~$0$ of~$(W_t)_{t\geq0}$.
We use this result to prove that the density $u(t,x,a)$ of the joint distribution~$\mu_t(x,a)$ of~$(X_t,A_t)_{t\geq0}$, defined on
$t\ge0$, $x\in\R$, $a\ge0$, satisfies, in a weak sense the
parabolic equation
\begin{equation}\label{eq:EDP}
   \partial_t u(t,x,a)
   =\mathcal L^*u(t,x,a)
   =a\partial^2_{xx}u(t,x,a)
   -\partial_a\left(f(a)u(t,x,a)\right)\delta_{x=0}, 
\end{equation}
with initial condition~$u_0(x,a)$. Here,~$\mathcal L^*$ is the adjoint
of the generator~$\mathcal L$ defined in~\eqref{eq:generateur_PRE}. 

Next we study the partial differential equation \eqref{eq:EDP} and the general questions we are concerned with are the following.
Does the joint distribution~$\mu_t$
at time~$t$, which is a measure, have a density with respect to the Lebesgue measure when
considering general initial condition~$\mu_{in}$? By studying the regularity of the solution to \eqref{eq:EDP}, do we recover the results observed during the probabilistic study? In particular, if $f(a)=-a^\gamma$ with $\gamma \ge 3/2$, can we prove global existence? In the case $\gamma \in (0,3/2)$ can we prove that the solution to \eqref{eq:EDP} becomes unbounded in finite time in any $L^p$ space (so-called blow-up)?

We describe the plan of the paper. In Section~\ref{sec:EDS} we build and study the correct
equation associated with~\eqref{eq:PRE}.
Section~\ref{sec:gen} is devoted to the computation of the generator of~$(X_t,L_t^W)_{t\geq0}$ from which we
deduce the weak formulation satisfied by the joint distribution
associated to~$(X_t,A_t)$ solution to the correct version of~\eqref{eq:PRE}. In Section~\ref{sec:PDE} we study equation~\eqref{eq:EDP}. 


\section{Mathematical study of a correct version of~\eqref{eq:PRE}}\label{sec:EDS}

In the system~\eqref{eq:PRE}, the term~$\delta_{X_t=0}$ does not
make sense. Intuitively, the term~$\delta_{X_t=0}\dx t$ should represent a measure on~$[0,\infty)$
to the set of zeros of the
process~$(X_t)_{t\geq0}$. This reminds the notion of local time whose
definition we recall here for completeness.

For any continuous local
martingale~$(M_t)_{t\geq0}$, one can define the \emph{local time at~$0$}
of~$(M_t)_{t\geq0}$ by:
\[
L_t^M
:=\lim_{\epsilon\to0}\frac1{2\epsilon}
\int_0^t\mathbf1_{|M_t|\leq\epsilon}\dx\left<M\right>_t,
\]
where the limit holds in probability. 
The local time satisfies the scaling property
\[
L_t^{\lambda M} =\lambda L_t^M \textrm{ a.s. for any }\lambda>0.
\]
In particular, the process~$(\lambda^{-1}L_t^{\lambda M})_{t\geq0}$ does
not actually depend
on~$\lambda$.
Since the process~$\left(L_t^M\right)_{t\geq 0}$ is
continuous and nondecreasing, we can associate to it a measure~$\dx
L_t^M$ without atoms on~$\R_+$. This measure
is supported by the
set~$\{t\ge 0:\, M_t=0 \}$.
For more details on the theory of local times, we refer to~\cite{revuz-yor-99}.

The formal term~$\delta_{X_t=0}\dx t$ should satisfy the
scaling invariance property
$\delta_{\lambda X_t=0}\dx t
=\delta_{X_t=0}\dx t$.
Comparing this formal property to the properties of local time, it seems natural
to replace the term~$\delta_{X_t=0} \dx  t$ by the renormalized local
time~$\dx L_t^X/2A_t$.
As a consequence, in this work, instead of~\eqref{eq:PRE}, we will study the system
\begin{equation}\label{eq:EDS}
\begin{cases}
\dx X_t&=\sqrt{2A_t}\dx W_t,\\
\dx A_t&=\displaystyle f(A_t)\frac{\dx L_t^X}{2A_t},
\end{cases}
\end{equation}
with a given initial condition~$(X_0,A_0)$. In the sequel, $f$ will be a locally
Lipschitz continuous function from~$(0,\infty)$ to~$\R$, and the initial
condition will be assumed to satisfy~$A_0>0$ almost surely.
Note that if $(X_t)_{t\geq0}$
solves~\eqref{eq:EDS}, then~$(X_t)_{t\geq0}$ is a local martingale.

More precisely, we are interested in proving that Equation~\eqref{eq:EDS} defines
a Markov process whose generator is given by~\eqref{eq:generateur_PRE}.
In order to do so we start by studying a simpler problem
\begin{equation}\label{eq:EDS_LW}
\begin{cases}
\dx X_t&=\sqrt{2A_t}\dx W_t,\\
\dx A_t&=\displaystyle f(A_t)\frac{\dx L_t^W}{\sqrt{2A_t}},
\end{cases}
\end{equation}
and we prove that the solutions to systems~\eqref{eq:EDS} and~\eqref{eq:EDS_LW} coincide when the initial condition
satisfies $X_0=0$. 

Equation~\eqref{eq:EDS} may not admit solutions for all positive
time, because solutions might blow up in finite time.
For example if~$f$ is a positive function, the process~$(A_t)_{t\geq0}$
will be nondecreasing, and nothing will a priori prevent it to go to infinity in
a finite time~$\tau$. In that case, the diffusion coefficient of~$(X_t)_{t\geq0}$ will blow
up in finite time, and~$(X_t)_{t\geq0}$ will not admit any extension
after time~$\tau$.

As a consequence, in the sequel, we will call a \emph{(strong)
  solution} to Equation \eqref{eq:EDS} (resp. \eqref{eq:EDS_LW}), a
triple~$(\tau,(X_t)_{0\leq t<\tau},(A_t)_{0\leq t<\tau})$, where~$\tau$
is a stopping time
of the Brownian motion~$(W_t)_{t\geq 0}$ and~$(X_t,A_t)_{0\leq t<\tau}$
is a continuous process adapted to~$(W_t)_{t\geq0}$ satisfying
Equation~\eqref{eq:EDS} (resp. \eqref{eq:EDS_LW}) until time~$\tau$.

We will say that such a solution is \emph{maximal},
when the process~$(A_t)_{0\leq t<\tau}$ converges either to~$0$
or~$\infty$ as~$t\to\tau$ on the event~$\{\tau<\infty\}$.
Indeed, in those two cases, the term~$f(A_t)$ appearing in the equation
becomes ill-defined at time~$\tau$, since~$f(a)$ is only assumed to make
sense for~$a\in(0,\infty)$.


\subsection{Well-posedness of~\eqref{eq:EDS_LW}}
The first equation in~\eqref{eq:EDS_LW} is explicit
in $(W_t,A_t)_{t\geq0}$: for given~$(X_0,(A_t)_{0\leq
  t<\tau},(W_t)_{t\geq0})$, its unique solution is given by
\[
\forall t<\tau, ~ X_t=X_0+\int_0^t\sqrt{2A_s}\dx W_s\, .
\]
Moreover, the second equation in~\eqref{eq:EDS_LW} is a closed
equation on~$(A_t)_{t\geq0}$ and does not depend on~$(X_t)_{t\geq0}$.
Thus, studying existence and uniqueness for the
equation~$\dx A_t=f(A_t)\dx L_t^W/\sqrt{2A_t}$ is
enough to obtain existence and uniqueness for system~\eqref{eq:EDS_LW}.

Recalling that~$f:(0,\infty)\to\R$ is assumed to be locally
Lipschitz continuous, we obtain the following result.

\begin{prop}\label{prop:processus_simple}
Let~$(X_0, A_0)$ be a random couple, independent of
$(W_t)_{t\geq0}$ and such that~$A_0>0$.  Then, there exists a unique maximal strong
solution $(\tau,(X_t)_{0\leq t\leq \tau},(A_t)_{0\leq t\leq \tau})$ to
Equation~\eqref{eq:EDS_LW}
with initial condition~$(X_0,A_0)$.
\end{prop}
\begin{proof}
As explained before, we only need to show existence and uniqueness for
the equation~$\dx A_t=f(A_t)\dx L_t^W/\sqrt{2A_t}$, which is closely
related to the
ordinary differential equation
$y'=f(y)/\sqrt{2y}$.

First, consider the flow~$\Phi$ associated with $y'=f(y)/\sqrt{2y}$.
Namely,~$t\mapsto\Phi_x(t)$ is the unique maximal solution
to~$y'=f(y)/\sqrt{2y}$ satisfying~$\Phi_x(0)=x$. This flow is well defined from the local
Lipschitz continuity of~$y\mapsto f(y)/\sqrt{2y}$.
The flow~$\Phi_x$ is only defined up to a time~$T(x)$. Moreover, in the
case~$T(x)<\infty$, one necessarily has
either~$\lim_{t\to T(x)}\Phi_x(t)=\infty$ or~$\lim_{t\to T(x)}\Phi_x(t)=0$. 

Then, one can check
than~$A_t=\Phi_{A_0}(L_t^W)$ is defined up to the
time~$\tau=\sup\{t\geq0, L_t^W<T(A_0)\}$ and satisfies the
equation $\dx A_t=f(A_t)\dx L_t^W/\sqrt{2A_t}$.
Indeed, since~$\Phi_x$ is continuously differentiable
and~$(L_t^W)_{t\geq0}$ is a contiuous nondecreasing process, then the usual chain
rule holds, namely one has $\dx\Phi_x(L_t^W)=\Phi_x'(L_t^W)\dx
L_t^W$. Moreover, on the event~$\{\tau<\infty\}$, $\lim_{t\to\tau}A_t$
exists with value~$0$ or~$\infty$.

For uniqueness of the solution to~$\dx A_t=f(A_t)\dx L_t^W/\sqrt{2A_t}$, consider
two solutions $(\tau,(A_t)_{0\leq t<\tau})$
and $(\tilde\tau,(\tilde A_t)_{0\leq t<\tilde\tau})$,
with~$A_0=\tilde A_0$. Then one has the following
Gr\"onwall-type inequality: $\forall t<\tau\wedge\tilde\tau$
\[
|A_t-\tilde A_t|
=\left|\int_0^tf(A_s)/\sqrt{2A_s}-f(\tilde A_s)/\sqrt{2\tilde A_s}\dx L_s^W\right|
\leq C\int_0^t|A_s-\tilde A_s|\dx L_s^W\, .
\]
Hence, from the expression
\begin{eqnarray*}
& &e^{-CL_t^W}\int_0^t\left|A_s-\tilde A_s\right|\dx L_s^W\\
&=&\int_0^te^{-CL_s^W}\left(
|A_s-\tilde A_s|-C\int_0^s\left|A_u-\tilde A_u\right|\dx L_u^W
\right)\dx L_s^W\\
&\leq &0\, ,
\end{eqnarray*}
it follows that~$|A_t-\tilde A_t|=0$ for~$\dx L_t^W$-almost all~$0\leq t<\tau\wedge\tilde\tau$.
\end{proof}

\begin{rem}
As it appears in the proof of Proposition~\ref{prop:processus_simple},
existence for the stochastic differential equation~\eqref{eq:EDS_LW} still holds true
provided the ordinary differential equation~$y'=f(y)/\sqrt{2y}$ admits a (non necessarily
unique) solution.
\end{rem}


\subsection{Link with system~\eqref{eq:EDS}}
In this section, we provide a link between solutions to
systems~\eqref{eq:EDS} and~\eqref{eq:EDS_LW}.
Consider a solution $(\tau,(X_t)_{0\leq t<\tau},(A_t)_{0\leq t<\tau})$
to~\eqref{eq:EDS} starting from~$X_0=0$.
We prove that the two processes~$(X_t)_{t\geq0}$ and~$(W_t)_{t\geq0}$ vanish at exactly
the same times, until the explosion time~$\tau$. Indeed, we show that the measure~$\dx L_t^X$ has a
density with respect to~$\dx L_t^W$.

\begin{prop}\label{prop:temps_loc}
Let~$(X_0,A_0)$ be a random variable independent
of~$(W_t)_{t\geq0}$ with $A_0>0$.
Let~$(\tau,(A_t)_{0\leq t<\tau},(X_t)_{0\leq t<\tau})$ be a strong solution to~\eqref{eq:EDS}. Then,
\[
X_t=\sqrt{2A_t}\left(W_t+\frac{X_0}{\sqrt{2A_0}}\right)\, .
\]
If in addition we assume that~$X_0=0$, then 
\begin{equation}\label{eq:forme_produit_0}
\forall t<\tau,~\dx L_t^X=\sqrt{2A_t}\dx L_t^W
~~\mbox{ and }~~
X_t=\sqrt{2A_t}W_t.
\end{equation}
\end{prop}

\begin{proof}
Let the stopping time~$\tau_n$ be the first positive
time for which~$\left(A_t\right)_{0\leq t<\tau}$ reaches
$[0,\frac1n]$. It is defined by
\[
\tau_n =\tau \wedge  \inf \left\{t \in [0,\tau)\, , A_t \leq \frac1n \right\}\, .
\]
Almost surely, $\tau_n \to \tau$ as $n$ goes to infinity
and $A_t>\frac1n$ on $[0,\tau_n)$.
Then, on the time interval~$[0,\tau_n)$, one has
\begin{align}\label{eq:forme_produit}
\dx W_t-\dx\left(\frac{X_t}{\sqrt{2A_t}}\right)
&=\left(\dx W_t-\frac{\dx X_t}{\sqrt{2A_t}}\right)
-X_t\dx\left(\frac1{\sqrt{2A_t}}\right)\\
&=0+(2A_t)^{-3/2}X_t\dx A_t\nonumber\\
&=0\, ,\nonumber
\end{align}
where the last equality is a consequence of $X_t\dx L_t^X=0$.
As a consequence, we deduce that
$W_t-\frac{X_t}{\sqrt{2A_t}}
=W_0-\frac{X_0}{\sqrt{2A_0}}
=-\frac{X_0}{\sqrt{2A_0}}$
up to time~$\tau_n$. Letting~$n$ go to
infinity, we
obtain~$X_t=\sqrt{2A_t}\left(W_t+\frac{X_0}{\sqrt{2A_0}}\right)$ for
all~$0\leq t<\tau$.

According to Tanaka's formula (see for example \cite{revuz-yor-99}),
the local time of a semimartingale is
given by~$\dx L_t^M=\dx|M_t|-\mathrm{sign}(M_t)dM_t$. Hence,
it the case~$X_0=0$, we get
\begin{align*}
\dx L_t^X
&=\dx|X_t|-\mathrm{sign}(X_t)\dx X_t
=\dx(\sqrt{2A_t}|W_t|)-\mathrm{sign}(W_t)\sqrt{2A_t}\dx W_t\\
&=\sqrt{2A_t}\Big(\dx|W_t|-\mathrm{sign}(W_t)\dx W_t\Big)
+(|W_t|-\textrm{sign}(W_t)W_t)\frac{\dx A_t}{\sqrt{2A_t}}\\
&=\sqrt{2A_t}\dx L_t^W.
\end{align*}
\end{proof}

Proposition~\ref{prop:temps_loc} is what we needed to establish a link
between solutions to~\eqref{eq:EDS} and~\eqref{eq:EDS_LW}.
\begin{cor}\label{cor:processus_simple}
A continuous process~$(X_t,A_t)_{0\leq t<\tau}$ defined up to
time~$\tau$ and satisfying~$X_0=0$ is a strong solution to
Equation~\eqref{eq:EDS_LW} if and only if it is a strong solution to
Equation~\eqref{eq:EDS}.
\end{cor}
\begin{proof}
This is a direct consequence of the second equality in~\eqref{eq:forme_produit_0}.
\end{proof}

\begin{cor}\label{cor:existence_unicite}
For any initial condition~$(X_0,A_0)$ independent of~$(W_t)_{t\geq0}$,
there exists a unique maximal
solution to Equation~\eqref{eq:EDS}.
\end{cor}
\begin{proof}
If~$X_0=0$, there exists a unique maximal solution to~\eqref{eq:EDS_LW}
from Proposition~\ref{prop:processus_simple}. From
Corollary~\ref{cor:processus_simple}, it is also the unique maximale
solution to~\eqref{eq:EDS_LW}.

For a general initial condition, up to the time~$\zeta=\inf\{t\geq0,X_t=0\}$,
system~\eqref{eq:EDS} clearly admits a unique
solution~$X_t=X_0+\sqrt{2A_0}W_t$, $A_t=A_0$. After $\zeta$, the Markov
property allows to apply existence and uniqueness starting from~$X_0=0$.
\end{proof}

\subsection{A discrete time approximation}\label{sec:approx}

In this section, we construct an approximation to the
process~$(X_t,A_t)_{t\geq0}$. This will give an heuristic justification to equation \eqref{eq:forme_produit_0}.

The Brownian motion will be discretized by a simple random walk
\[
Y_n=\sum_{k=1}^nU_k,~~n\in{\N},
\]
where $(U_k)_{k\in\N^*}$ is a sequence of independant random variables,
uniformly distributed on $\{1,-1\}$.
We will need the  \emph{discrete local time} of~$(Y_n)_{n\in\N}$, defined by
\[
\Lambda_n
=\sum_{k=1}^n\mathbf1_{Y_k=0}.
\]
Let~$t>0$. Then one has the convergence in distribution
\[
\left(\sqrt{\frac tn}Y_n,\sqrt{\frac tn}\Lambda_n\right)
\to(W_t,L_t^W)\, ,
\]
as~$n\to\infty$.


Two different approximations are natural. Here we modify the step size as in \cite{ben-naim-redner-04}.
\[
\forall k\in\{1,\hdots,n\}
\begin{cases}
\hat X_k^n&=\hat X_{k-1}^n+\sqrt{2\hat A_{k-1}^n\frac tn}U_n,\\
\hat A_k^n&=\hat A_{k-1}^n+f(\hat A_{k-1}^n)\sqrt{\frac tn}\mathbf1_{\hat X_k^n=0}\, .
\end{cases}
\]

We first state a discrete analog to Equation~\eqref{eq:forme_produit_0}.
\begin{lem}
One has the equalities
\[
X_k^n=\sqrt{2\hat A_k^n\frac tn}Y_k,
\mbox{ and }
\mathbf1_{\hat X_k^n=0}=\mathbf1_{Y_k=0}\, .
\]
\end{lem}

Consider the Euler scheme associated to~$y'=f(y)$, namely
\[
y_{n+1}^\delta=y_n^\delta+\delta f(y_n^\delta)
\]
where~$\delta$ is some time step. Then, when~$n\to\infty$
and~$\delta\to0$ in the regime~$n\delta\to t$, one has
\[
y_n^\delta\to\Phi_{y_0}(t)\, .
\]

\begin{lem}
$\hat A_k^n$ is given by~$y_{\Lambda_k}^{\sqrt{t/n}}$.
\end{lem}

\begin{theo}
As $n$ goes to infinity, $(\hat X_n^n,\hat A_n^n)$ converges in
distribution to~$(X_t,A_t)$.
\end{theo}
\begin{proof}
One has~$\hat A_n^n=y_{\Lambda_n}^{\sqrt{t/n}}$.
Here~$\sqrt{\frac tn}$ and $\Lambda_n$ respectively converge to~$0$
and~$\infty$ in the regime~$\sqrt{\frac tn}\Lambda_n\to L_t^W$. As a
consequence,~$\hat A_n^n$ converges to~$\Phi_{A_0}(L_t^W)=A_t$.

On the other hand, one has
\[
\hat X_n^n
=\sqrt{2\hat A_n^n}\times\sqrt{\frac tn}U_n
\to\sqrt{2A_t}\times W_t
=X_t\, .
\]
\end{proof}

\begin{rem}
For simplicity, we only proved convergence for a fixed time~$t$.
Actually one can prove convergence for the whole trajectory.
\end{rem}

%
%

\subsection{A particular case:~$f$ is a power function}\label{sec:particular_case}
From the second equality in~\eqref{eq:forme_produit_0}, one can expect at least
four different long-time behaviors for the
process~$(X_t)_{t\geq0}$. Indeed, the process can stop in finite
time, in the case where there exists a finite time $t$ such that~$A_t=0$. On the opposite, the process can
perform very large oscillations if~$(A_t)_{t\geq0}$ tends to infinity in finite time. 
Last, when~$(A_t)_{t\geq0}$ takes its values in~$(0,\infty)$, we can expect the
process~$(X_t)_{t\geq0}$ either to go asymptotically to~$0$,
if~$(A_t)_{t\geq0}$ decreases fast enough, or to be recurrent in~$\R$,
in the case where~$(A_t)_{t\geq0}$ remains large enough.

Those four behaviors actually do occur in the case of a power
function~$f(a)=\pm a^\gamma$. The advantage of such a function is that
the expression of~$A_t$ can explicitely be computed as a function of~$L_t^W$. 

When~$(A_t)_{t\geq0}$
remains in~$(0,\infty)$ for all positive~$t$, one can derive a
polynomial behavior for~$(X_t)_{t\geq0}$. Indeed, up to renormalisation
by a power of~$t$, we prove convergence in law to a non-Gaussian
distribution for~$(X_t)_{t\geq0}$.

In a first time, we give an explicit expression of $A_t$, and then we
give the asymptotic behavior of $X_t$, depending on the sign of $f$.

\subsubsection{Explicit of $A_t$}
The following lemma gives the expression of~$A_t$ as a function of~$L_t^W$.
For simplicity  we assume that~$X_0=1$ and~$A_0=1$.
\begin{lem}\label{lem:expression_explicite_A}
Let~$f(a)=\sigma a^\gamma$, with~$\sigma =\pm 1$.
The solution~$(X_t,A_t)_{0\leq t<\tau}$ to~\eqref{eq:EDS} exists up to the time~$\tau$ defined by
\begin{equation}\label{eq:def_tau}
\tau:=
\begin{cases}
+\infty&
\textrm{ if }\sigma(3/2-\gamma)\ge 0,\\
\inf
\left\{
t\geq0,~L_t^W=\frac{\sqrt2}{\sigma(\gamma-3/2)}
\right\}&
 \textrm{ if } \sigma(3/2-\gamma) <0 \, .
\end{cases}
\end{equation}
For all $t\in [0,\tau)$, one has
\[
A_t=
\begin{cases}
e^{\sigma L_t^W/\sqrt2}
& \textrm{ if } \gamma =3/2 \, ,\\
(1+\sigma/\sqrt2(3/2-\gamma)L_t^W)^{\frac1{3/2-\gamma}}
& \textrm{ if } \gamma \neq 3/2 \, .
\end{cases}
\]
\end{lem}
\begin{proof}
As we have seen in the proof of Proposition~\ref{prop:processus_simple}, $A_t$
is given by~$A_t=\Phi_{A_0}(L_t^W)$ where~$\Phi_x$ is the solution to~$y'(t)=f(y(t))/\sqrt{2y}$ with initial condition~$\Phi_x(0)=x$.
Here,~$\Phi_1$ is given by
\[
\Phi_1(t)=
\begin{cases}
e^{\sigma t/\sqrt2}& \textrm{ if } \gamma =3/2 \, ,\\
(1+\sigma/\sqrt2(3/2-\gamma)t)^{\frac1{3/2-\gamma}}
& \textrm{ if } \gamma \neq 3/2\, ,
\end{cases}
\]
which lies in~$(0,\infty)$ for~$0\leq t<\sqrt2(\sigma(\gamma-3/2))^{-1}$
if~$\sigma(3/2-\gamma)<0$, and for all~$t\geq0$ otherwise. The expressions
of~$\tau$ and~$A_t$ easily follow.
\end{proof}


\subsubsection{Local deceleration: $f(a)=- a^\gamma$}
In that case the process~$(X_t)_{t\geq0}$ is slowing down
and we obtain a dynamical transition to an absorbing state. For
sufficiently strong deceleration, the particle might get trapped at the
origin in finite time while if the deceleration process is
sufficiently weak the particle never gets
trapped.
\begin{prop}\label{prop:explo_1}
Assume that $f(a)=- a^\gamma$, then 
\begin{itemize}
\item
  if~$\gamma<3/2$, the stopping
  time~$\tau$ defined in \eqref{eq:def_tau} is almost surely finite, and one
  has $\lim_{t\to\tau}(X_t,A_t)=(0,0)$;
\item
  if~$3/2\leq\gamma<2$, $\tau=\infty$ and $X_t \to 0$ as~$t$ goes
  to~$\infty$. However, almost surely, for all~$t>0$, there exists~$s>t$
  such that~$X_s\neq0$;
\item
  if~$2\leq\gamma$, then $\tau=\infty$ and the process~$(X_t)_{t\geq0}$ is
  recurrent in~$\R$.
\end{itemize}
\end{prop}
\begin{proof}
The case~$\gamma<3/2$ is a consequence of the fact that~$(A_t)_{t\geq0}$
is absorbed by~$0$ in finite time.

The case~$\gamma\geq3/2$,~$\gamma\neq2$ follows from the almost sure
asymptotic behavior~$t^{1/2-\epsilon}=o(L_t^W)$
and~$L_t^W=o(t^{1/2+\epsilon})$, for any~$\epsilon>0$, and from the law
of iterated logarithm for Brownian motion.

In, the limit case~$\gamma=2$,~$(X_t)_{t\geq0}$ is equivalent in
the long time to~$(8W_t/L_t^W)_{t\geq0}$, and it is enough to consider
this latter process.
From \cite{khoshnevisan-96}, Theorem 4.5, in the case~$\beta=2$,
and~$f(x)=1/(x\log x)$, there exist random times~$(t_n)_{n\in\N}$
with~$t_n\to\infty$ such that
\[
\forall n\in\N,~\frac{\sup_{s\leq t_n}|W_s|}{L_{t_n}^W}\geq \log(L_{t_n}^W)\, .
\]
Let~$(\tilde t_n)_{n\in\N}$ be a nondecreasing sequence
satisfying~$W_{\tilde t_n}=\sup_{s\leq t_n}|W_s|$, in particular one
has~$|W_{\tilde t_n}|/L_{\tilde t_n}^W\to\infty$.
Since~$\limsup_{t\to\infty}|W_t|=\infty$, up to a subsequence, one
has~$\tilde t_n\to\infty$.
Moreover, one can find another sequence of random times~$t'_n$ going
to~$\infty$ and satisfying~$W_{t'_n}=0$.
As a consequence, one obtains
\[
\liminf_{t\to\infty}\frac{|W_t|}{L_t^W}=0
\mbox{ and }
\limsup_{t\to\infty}\frac{|W_t|}{L_t^W}=\infty\, .
\]
Hence,~$(|X_t|)_{t\geq0}$ is recurrent in~$[0,\infty)$, and
by symmetry this implies that~$(X_t)_{t\geq0}$ is recurrent in~$\R$.
\end{proof}

Using the equality in distribution
\begin{align}\label{eq:egalite_en_loi}
X_t
=\sqrt{2A_t}W_t
&=\sqrt2\left(1-(3/2-\gamma)/\sqrt2L_t^W\right)^{\frac1{3-2\gamma}}W_t\nonumber\\
&\stackrel{(d)}=
\sqrt2\left(1+(\gamma-3/2)\sqrt{t/2}L_1^W\right)^{\frac1{3-2\gamma}}\sqrt tW_1\, ,
\end{align}
 for $\gamma\neq3/2$, we deduce that, when~$\gamma>1$, the decreasing rate of the process~$(X_t)_{t\geq0}$ is~$t^{\frac{2-\gamma}{3-2\gamma}}$. In this expresson of the rate, the exponent may be nonpositive or nonnegative. More
precisely, one has:
\begin{prop}
If $f(a)=- a^\gamma$, with $\gamma>3/2$, then the convergence in distribution, 
\[
t^{\frac{\gamma-2}{3-2\gamma}}X_t
\stackrel{(d)}\to C_\gamma (L_1^W)^{\frac1{2(3-2\gamma)}}W_1 \quad \textrm{as } t\to\infty \, ,
\]
holds true where~$C_\gamma=2^{\frac{1-\gamma}{3-2\gamma}}(\gamma-3/2)^{\frac1{2(3-2\gamma)}}$.
\end{prop}

One can also give the asymptotic behavior as~$t\to\tau$ when $\tau<\infty$.
\begin{prop}
If $f(a)=- a^\gamma$, with $\gamma<3/2$, then, as~$t\to0$,
\[
\mathbf1_{t<\tau}t^{\frac{\gamma-2}{3-2\gamma}}X_{\tau-t}
\stackrel{(d)}\to C_\gamma' (L_1^W)^{\frac1{2(3-2\gamma)}}W_1\, ,
\]
where~$C_\gamma'=2^{\frac{1-\gamma}{3-2\gamma}}(3/2-\gamma)^{\frac1{2(3-2\gamma)}}$.
\end{prop}
\begin{proof}
We use the reversibility property of the Brownian
motion that we recall,  for~$T>0$,  setting~$\zeta = \inf\{t>0,L_t^W=T\}$, the
equality in distribution 
\[
(W_t,L_t^W)_{0\leq t\leq\zeta}
\stackrel{(d)}=
(W_{\zeta-t},T-L_{\zeta-t}^W)_{0\leq t\leq\zeta}\, ,
\]
holds true. In our situation, we apply this property to the stopping
time~$\zeta=\tau$, and we then use \eqref{eq:egalite_en_loi}.
\end{proof}

\begin{rem}
The fact that $(X_t)_{t\geq0}$ can be trapped at~$0$ for~$\gamma<3/2$ was already noticed
in~\cite{benichou-meunier-redner-voituriez-12}. However, the
different behavior for~$\gamma\in[3/2,2)$ was not observed.
\end{rem}
\begin{rem}
As a consequence of Lemma~\ref{lem:expression_explicite_A}, the
survival probability of~$(X_t)_{t\geq0}$ at time~$t$ is given
for~$\gamma<3/2$ by
\[
S(t)
=\PP\left(L_t^W\leq\frac{\sqrt2}{3/2-\gamma}\right)
=\PP\left(|W_1|\leq\frac{\sqrt2}{\sqrt t(3/2-\gamma)}\right)
\underset{t\to\infty}\sim\frac2{(3/2-\gamma)\sqrt{\pi t}}\, ,
\]
where we used the equalities in distribution~$L_t^W=|W_t|=\sqrt t|W_1|$.
This fact was already observed
in~\cite{benichou-meunier-redner-voituriez-12}.
\end{rem}


\subsubsection{Local acceleration: $f(a)= a^\gamma$}
In such a case the diffusion coefficient of~$(X_t)_{t\geq0}$ is
nondecreasing. Again, the proof of the following result follows from
Lemma~\ref{lem:expression_explicite_A} and the relation~$X_t=\sqrt{ 2A_t}W_t$.
\begin{prop}
Assume that $f(a)= a^\gamma$, then
\begin{itemize}
\item if~$\gamma\geq2$, the stopping
  time $\tau$ defined
  in~\eqref{eq:def_tau} is almost surely finite and~$\lim_{t\to\tau}A_t=\infty$.
  Moreover,~$\lim_{t\to\tau}X_t=0$;
\item if~$3/2<\gamma<2$,  $\tau$ is almost surely finite and~$\lim_{t\to\tau}A_t=\infty$.
  Moreover, $\liminf_{t\to \tau}X_t=-\infty$ and $\limsup_{t\to \tau}X_t=\infty$;
\item if~$\gamma\le 3/2$, the
  time~$\tau$ satisfies $\tau=\infty$, and
  $A_t \to \infty$ when $t\to \infty$.
\end{itemize}
\end{prop}
Furthermore, when~$\gamma<3/2$, from equality in
distribution~\eqref{eq:egalite_en_loi}, one can deduce that the long
time behavior of~$(X_t)_{t\geq0}$ is of order $t^{\frac{2-\gamma}{3-2\gamma}}$,
where the exponent is positive.
More precisely the following result holds true.
\begin{prop}
If $f(a)=a^\gamma$, with $\gamma<1$, then, as~$t\to\infty$, one has the
convergence in distribution
\[
t^{\frac{\gamma-2}{3-2\gamma}} X_t
\stackrel{(d)}\to C'_\gamma(L_1^W)^{\frac1{3-2\gamma}}W_1\, ,
\]
where~$C'_\gamma=2^{\frac{1-\gamma}{3-2\gamma}}(3/2-\gamma)^{\frac1{2(3-2\gamma)}}$.
\end{prop}

We can also describe the rate of explosion of~$(X_t,A_t)_{t\geq0}$ as
goes to~$\tau$, in the case~$\tau<\infty$.
\begin{prop}
If $f(a)=a^\gamma$, with $\gamma>1$, then, as~$t\to0$, one has the
convergence in distribution
\[
\mathbf1_{t<\tau}t^{\frac{\gamma-2}{3-2\gamma}} X_{\tau-t}
\stackrel{(d)}\to C_\gamma(L_1^W)^{\frac1{3-2\gamma}}W_1,
\]
where~$C_\gamma=2^{\frac{1-\gamma}{3-2\gamma}}(\gamma-3/2)^{\frac1{2(3-2\gamma)}}$.
\end{prop}

\begin{rem}
The case~$\gamma=0$ was treated by deterministic methods in
\cite{benichou-meunier-redner-voituriez-12}, through an approximation of
the Laplace transform of the distribution of~$X_t$. Here, by using the
stochastic
differential equation~\eqref{eq:EDS}, we were able to
compute the exact asymptotic behavior. We obtain that the growth rate
of~$(X_t)_{t\geq0}$ is given
by~$t^{2/3}$, and that its diffusion coefficient, given by~$\left(L_t^W\right)^{3/2}$,
behaves as~$t^{1/3}$.
Those exponents were correctly predicted in
\cite{benichou-meunier-redner-voituriez-12}.
\end{rem}


\section{Generator of the process~$(X_t,A_t)_{t\geq0}$}\label{sec:gen}

In this section we investigate the generator of the Markov
process~$(X_t,A_t)_{t\geq0}$ solution to system~\eqref{eq:EDS}. The
expression of the generator will follow from the
generator of the process $(W_t,L_t^W)_{t\geq0}$, where $W_t$ is a standard Brownian motion
and $L_t^W$ is its local time at $0$. 

Consider the unique maximal solution~$(\tau,(X_t)_{0\leq t<\tau},(A_t)_{0\leq t<\tau})$ of~\eqref{eq:EDS}, whose existence is ensured by
Corollary~\ref{cor:existence_unicite}. The first step is to extend its state space in order to define a continuous Markov process for all positive times.

\subsection{Extended state space}
In the proof of Proposition~\ref{prop:processus_simple}, we mentionned
that, when~$\tau$ is finite,~$A_t$ necessarily converges as~$t$ goes
to~$\tau$, either toward~$0$ or~$\infty$. In the case~$A_t\to0$, one can
also determine the behavior of~$X_t$, as stated in the following lemma.
\begin{lem}\label{lem:A_t->0}
On the event~$\{\tau<\infty, \lim_{t\to\tau}A_t=0\}$, $X_t$
converges to~$0$ as~$t$ goes to $\tau$.
\end{lem}
\begin{proof}
First, one notices that on the
set~$\left\{\tau<\infty, \lim_{t\to\tau}A_t=0\right\}$
the limit
\[
\lim_{t\to\tau}X_t=X_0+\int_0^\tau\sqrt{2A_s}\dx W_s
\]
exists almost surely. It is thus enough to prove that the event
$E=\{\tau < \infty, \lim_{t\to\tau}A_t=0,\lim_{t\to\tau}X_t\neq0\}$ is a
null set.

On $E$, there exists a random variable
$h>0$ such that $X_t \neq 0$ for all
$t\in (\tau-h,\tau)$. Hence, on $E$, one has for
all $t\in (\tau-h,\tau)$
\[
A_t-A_\tau
=\int_t^\tau \dx A_s
=\int_t^\tau f(A_s)\dx L_t^X
=0\, .
\]
However, on~$E$, one also has $\tau=\inf\{t\geq0, A_t=0\}$, which
contradicts the fact that~$A_t$ is constant on~$(\tau-h,\tau)$.
As a consequence, $E$ has probability~$0$,
which concludes the proof.
\end{proof}

From Lemma~\ref{lem:A_t->0}, the maximal
solution~$(\tau,(X_t)_{0\leq t<\tau},(A_t)_{0\leq t<\tau})$ to
Equation~\eqref{eq:EDS} can be
extended to a process defined for all positive
times by setting
\[
(X_t,A_t)=
\begin{cases}
(0,0)&\hbox{ on }\left\{\tau\leq t, \lim_{s\to\tau}A_s=0\right\},\\
(0,\infty)&\hbox{ on }\left\{\tau\leq t, \lim_{s\to\tau}A_s=\infty\right\}\, .
\end{cases}
\]
For notational simpliciy the extended process will still be denoted
by $(X_t,A_t)_{t\geq0}$. This will define a Markov process with state
space~$\mathcal E=\left(\R\times[0,\infty)\right)\cup\{(0,\infty)\}$,
which is the half plane~$\R\times[0,\infty)$ augmented with an
additional point~$(0,\infty)$. We define the following topology
on~$\mathcal E$: the subset~$\R\times[0,\infty)$ is endowed with its
usual topology, and we choose the family $(\R\times[\alpha,\infty))_{\alpha>0}$
as a neighborhood basis of~$(0,\infty)$.
In other words, any sequence~$(x_n,a_n)_{n\in\N}$
in~$\R\times[0,\infty)$ with~$a_n\to\infty$ will
satisfy~$(x_n,a_n)\to(0,\infty)$ in~$\mathcal E$.

With these conventions,~$(X_t,A_t)_{t\geq0}$ defines a continuous Markov
process with values in~$\mathcal E$, defined for all positive times. A
natural question is then to investigate its generator, and to compute
the distribution of~$(X_t,A_t)$ for a given~$t>0$.
Note that the two points~$(0,0)$ and~$(0,\infty)$ are absorbing points
for the Markov process~$(X_t,A_t)_{t\geq0}$.


\subsection{Generator of the process~$(W_t,L_t^W)_{t\geq0}$}

\begin{prop}\label{prop:generateur}
The generator $\mathcal L_0$ of the Markov process~$(W_t,L_t^W)_{t\geq0}$ is given by
\[
\mathcal L_0h(w,l)=\frac12\partial_{ww}^2h(w,l)+\partial_lh(w,l) \delta_{w=0}\, .
\]
\end{prop}
\begin{rem}
Since the coefficient~$\delta_{w=0}$ is singular, the previous definition of the generator has to
be understood in a weak sense.
If~$\phi:\R\to\R$ is a continuous bounded function with
bounded support, then,
for~$h$ continuously differentiable in the $l$-variable and twice
continuously differentiable in the $w$-variable with bounded
derivatives, for all~$l\in[0,\infty)$, one has
\begin{eqnarray*}
\lim_{t\to0}
\int_\R\phi(w)\E^{w,l}\left[\frac{h(W_t, L_t^W)-h(w,l)}t\right]\dx w
&=&\int_\R\phi(w)\frac12\partial_{ww}^2h(w,l)\dx w
\\
& &+\, \phi(0)\partial_lh(0,l)\, .
\end{eqnarray*}
Here,~$\E^{w,l}$ stands for the expectation conditionaly
to~$\{W_0=w,~L_0^W=l\}$.
Equivalently, one has, in the distributional sense,
\[
\lim_{t\to0}
\E^{w,l}\left[\frac{h(W_t, L_t^W)-h(w,l)}t\right]
=\frac12\partial_{ww}^2h(w,l)
+\partial_lh(0,l)\delta_{w=0}\, .
\]
\end{rem}
\begin{proof}
  From time invariance of~$(L_t^W)_{t\geq 0}$, one can assume that~$l=0$.
  From the equality
  $h(W_t,L_t^W)=h(W_t,0)+L_t^W\int_0^1\partial_lh(W_t,sL_t^W)\dx s$, one
  obtains 
  \begin{align*}
    \E^{w,0}\left[\frac{h(W_t,L_t^W)-h(w,0)}t\right]
    &=\E^{w,0}\left[\frac{h(W_t,0)-h(w,0)}t\right]\\
    &\quad+\E^{w,0}\left[
      \frac{L_t^W}t\int_0^1\partial_lh(W_t,sL_t^W)\dx s
    \right]\, .
  \end{align*}
  The first term in the right hand side converges
  to~$\frac12\partial_{ww}^2 h(w,0)$ as~$t\to0$,
  since~$\frac12\partial_{ww}^2$ is the generator of the Brownian motion,
  using the boundedness of~$\partial_{ww}^2h$.

  Then, using the fact that the law of $(W_t,L_t^W)$ under~$\E^{w,0}$ is the
  same than the law of $(\sqrt tW_1,\sqrt tL_1^W)$ under~$\E^{w/\sqrt t,0}$,
  one gets
  \begin{align*}
    &\int_\R\phi(w)
    \E^{w,0}\left[
      \frac{L_t^W}t\int_0^1\partial_lh(W_t,sL_t^W)\dx s
    \right]\dx w\\
    =&\int_\R\phi(w)
    \E^{w/\sqrt t,0}\left[
      \frac{\sqrt tL_1^W}t\int_0^1\partial_lh(\sqrt tW_1,s\sqrt tL_1^W)\dx s
    \right]\dx w\\
    =&\int_\R\phi(\sqrt tw)
    \E^{w,0}\left[
      L_1^W\int_0^1\partial_lh(\sqrt tW_1,s\sqrt tL_1^W)\dx s
    \right]\dx w\, .
  \end{align*}
  From the dominated convergence theorem, this converges as~$t$ goes
  to~$0$ to
  \[
  \phi(0)\partial_lh(0,0)\int_\R\E^{w,0}[L_1^W]\dx w\, .
  \]
  Finally, from the occupation time formulation~$\int_\R L_t^{W+w}\dx w=t$
  (see \cite{revuz-yor-99}) one obtains
  \[
  \int_\R\E^{w,0}[L_1^W]\dx w=\int_\R\E^{0,0}[L_1^{W+w}]\dx w=1\, .
  \]
\end{proof}


\subsection{Generator of the process defined by system~\eqref{eq:EDS}}

Since~$(X_t,A_t)_{t\geq0}$ can be obtained as a function
of~$(W_t,L_t^W)_{t\geq0}$, we can compute the generator of the former
from the generator of the latter.

\begin{prop}\label{prop:generateur_X,A}
The generator~$\mathcal L$ of the Markov process~$(X_t,A_t)_{t\geq0}$ solution to system~\eqref{eq:EDS} is given by
\[
\mathcal Lh(x,a)=a\partial_{xx}^2h(x,a)+f(a)\partial_ah(x,a) \delta_{x=0}\, .
\]
\end{prop}
\begin{rem}\label{rem:faible}
Again, the previous definition has to be understood in a weak sense. If~$\phi:\R\to\R$ is a continuous bounded function with
bounded support, then, for~$h$ continuously differentiable in the
$a$-variable and twice continuously differentiable in the $x$-variable
with bounded derivatives, for all~$a\in[0,\infty)$ one
obtains
\begin{eqnarray*}
& &\lim_{t\to0}
\int_\R\phi(x)\tilde\E^{x,a}
\left[\frac{h(X_t, A_t)-h(x,a)}t\right]\dx x\\
&=&\int_\R\phi(x)a\partial_{xx}^2h(x,a)\dx x
+\phi(0)f(a)\partial_ah(0,a)\, .
\end{eqnarray*}
Here,~$\tilde\E^{x,a}$ stands for the expectation conditionaly to~$\{X_0=x,~A_0=a\}$.
\end{rem}
\begin{rem}\label{rem:eq_faible}
Since~$(X_t,A_t)$ is a Markov process, the following identity holds for
any probability density~$\phi$
\[
\int_\R\phi(x)\tilde\E^{x,a}[h(X_t,A_t)]\dx x
=\tilde\E^{\phi,a}[h(X_t,A_t)],
\]
where $\tilde E^{\phi,a}$ denotes the expectation for an initial condition
satisfying $A_0=a$ and such that $X_0$ admits $\phi$ as density.
In other word, when $X_0$ admits a continuous
density $v_0$, replacing~$\phi$ by $v_t$ in Remark \ref{rem:faible} yields the
following time-derivative at $t=0$:
\begin{eqnarray*}
& &\partial_t\tilde\E^{v_0,a}[h(X_t,A_t)]
=\partial_t\iint_{\R\times\R^+} h(x,a)\dx\mu_t(x,a)\\
&=&\int_{\R} a\partial_{xx}^2h(x,a)v_0(x)\dx x
+v_0(0)f(a)\partial_ah(0,a).
\end{eqnarray*}
This is a weak formulation of Equation \eqref{eq:EDP2} below. 
\end{rem}
\begin{proof}
Let~$(W_t)_{t\geq0}$ be a
Brownian motion started at $X_0/\sqrt{2A_0}$. 
From the proof of Proposition~\ref{prop:processus_simple}, we know that the
process~$(X_t,A_t)_{0\leq t<\tau}$ is given by
\[
\forall 0\leq t<\tau,~
\begin{cases}
X_t&=\sqrt{2\Phi_{A_0}(L_t^W)}W_t\, ,\\
A_t&=\Phi_{A_0}(L_t^W)\, ,
\end{cases}
\]
where~$t\to\Phi_x(t)$ is the flow of the differential equation~$y'=f(y)/\sqrt{2y}$
with initial condition $x$.
Then, setting $x=\sqrt{2a}w$, one has
(if we still denote by~$\E^{w,l}$ the expectation conditionally to~$\{W_0=w,L_0^W=l\}$)
\begin{align*}
&\int_\R\phi(x)\tilde\E^{x,a}
\left[\frac{h(X_t,A_t)-h(x,a)}t\right]\dx x\\
=&\int_\R\phi(x)\E^{\frac x{\sqrt{2a}},0}
\left[\frac{h(\sqrt{2\Phi_a(L_t^W)}W_t,\Phi_a(L_t^W))-h(x,a)}t\right]\dx x\\
=&\sqrt{2a}\int_\R\phi(\sqrt{2a}w)\E^{w,0}
\left[\frac{h(\sqrt{2\Phi_a(L_t^W)}W_t,\Phi_a(L_t^W))-h(\sqrt{2a}w,a)}t\right]\dx w\, .
\end{align*}
Applying Proposition~\ref{prop:generateur} to the function
$F(w,l)=h\left(\sqrt{2\Phi_a(l)}w,\Phi_a(l)\right)$ one obtains
\begin{align*}
&\lim_{t\to0}\int_\R\phi(x)\tilde\E^{x,a}
\left[\frac{h(X_t,A_t)-h(x,a)}t\right]\dx x\\
=&\lim_{t\to0}\sqrt{2a}\int_\R\phi(\sqrt{2a} w)\E^{w,0}
\left[\frac{F(W_t,L_t^W)-F(w,0)}t\right]\dx w\\
=&\sqrt{\frac a2}\int_\R\phi(\sqrt{2a}w)\partial_{ww}^2F(w,0)\dx w
+\sqrt{2a}\phi(0)\partial_lF(0,0)\\
=&\sqrt{\frac a2}\int_\R\phi(\sqrt{2a}w)(2a)\partial_{xx}^2h(\sqrt{2a}w,a)\dx w
+\Phi_a'(0)\sqrt{2a}\phi(0)\partial_ah(0,a)\\
=&\int_\R\phi(x)a\partial_{xx}^2h(x,a)\dx x
+f(a)\phi(0)\partial_ah(0,a)\, .
\end{align*}
\end{proof}


\section{PDE for the joint distribution}\label{sec:PDE}

As a consequence of Proposition~\ref{prop:generateur_X,A} and Remark~\ref{rem:eq_faible}, the density of the joint distribution~$\mu_t(x,a)$
associated to the process~$(X_t, A_t)_{t\ge 0}$ solution
to~\eqref{eq:EDS}, satisfies in a weak sense the following equation for $t>0$, $(x,a) \in \R \times \R_+ $: 
\begin{equation}\label{eq:EDP2}
\partial_t u(t,x,a) = a \partial^2_{xx} u(t,x,a) - \partial_a \left( f(a) \, u(t,x,a)\right) \delta_{x=0}\, ,  
\end{equation}
together with the initial condition:
\[u(t=0,x,a)=u_0(x,a)\, , \quad  (x,a) \in \R \times \R_+\, .
\]

In this Section, we first give the general form of the distribution~$\mu_t$ at
time~$t$. In particular we prove that the measure~$\mu_t$ has a density with respect to the Lebesgue measure when
considering general initial condition. Then,
we prove uniqueness of the solution to~\eqref{eq:EDP2}.
Furthermore, if $f(a)=-a^\gamma$, by studying the regularity of the solution in a $L^p$ framework, we recover the results observed during the probabilistic study: global existence if $\gamma \ge 3/2$, while, in the case $\gamma<3/2$, the solution becomes unbounded in finite time (so-called blow-up). Finally, as in~\cite{benichou-meunier-redner-voituriez-12}, using Laplace and Fourier transforms, for a particular initial condition we explicitely compute the solution to~\eqref{eq:EDP2}. 

\begin{rem}
Since the generator maps continuous function to measure it is possible to write the pde for the density of the joint distribution. In further works it would be interesting to investigate how the pde could be written for joint distribution which involves atoms.
\end{rem}


\subsection{Shape of the distribution of~$(X_t,A_t)$}

The three lemmas below give the general form of the distribution~$\mu_t$ at
time~$t$ of the solution~$(X_t,A_t)_{t\geq0}$ to~\eqref{eq:EDS}
starting from an initial condition~$(x_0,a_0)\in\R\times(0,\infty)$. In
particular, one starts from~$\mu_0=\delta_{(x_0,a_0)}$.

Recall that the two points~$(0,0)$ and~$(0,\infty)$ are absorbing points
for the process~$(X_t,A_t)_{t\geq0}$, so that the distribution starting
from those points will be constant, equal to~$\delta_{(0,0)}$ or
$\delta_{(0,\infty)}$ respectively.

By linearity of Equation~\eqref{eq:EDP2}, the distribution starting from a more
general initial condition can be obtained
as a mixture of distributions starting at deterministic points. We first consider the case where the initial condition is~$\mu_0=\delta_{(0,a_0)}$, with $a_0>0$. 

\begin{lem}\label{lem:ansatz_loi}
Assume that the initial condition~$(x_0,a_0)$ satisfies $f(a_0)\neq 0$
and $x_0=0$. Then, for all~$t>0$, there
exist a measurable
function $n_t\, : \, \R\times(0,\infty) \to \R $ and two real numbers $p_t\in[0,1]$
and  $q_t\in[0,1]$ such that
\begin{equation}\label{eq:forme_loi}
\mu_t( x, a)=
 n_t(x,a) \dx x \dx a
 +p_t \delta_{(0,0)}
 +q_t\delta_{(0,\infty)},
\end{equation}
where~$p_t+q_t+\int_{\R\times[0,\infty)}n_t(x,a)\dx x\dx a=1$.
Furthermore, if $f$ is nonnegative, one has $p_t=0$ for all $t>0$,
while if $f$ is nonpositive, one has $q_t=0$ for all $t>0$.
\end{lem}
\begin{proof}
All we need to prove is that the restriction of $\mu_t$ to the
set~$\R\times(0,\infty)$ admits a density with respect to the Lebesgue
measure.

First, since $t^{-1/2}(W_t,L_t^W)$
has the same distribution as $(W_1,L_1^W)$, which has a density, see
\cite{ito-mckean-74}, page 45, it follows that $(W_t,L_t^W)$ with $W_0=0$ admits a
density $\gamma_t$ on $\R \times (0,\infty)$. Moreover, as~$(X_t)_{0\leq t<\tau}$ starts
from the initial condition~$x_0=0$, Proposition
\ref{prop:temps_loc} states that 
$(X_t,A_t)= \left( \sqrt{2\Phi_{a_0}(L_t^W)}W_t, \Phi_{a_0} (L_t^W)\right)$,
where~$\Phi_{a_0}$ is the flow of the differential
equation~$y'=f(y)$ starting at~$a_0$.
For a given~$a_0$,~$\Phi_{a_0}(l)$ is defined for all~$l\in[0,T(a_0))$
for some~$T(a_0)\in[0,\infty]$.

Let us define the mapping
$\Psi(w,l):= \left(\sqrt{2\Phi_{a_0}(l)}w, \Phi_{a_0}(l)\right)$,
for~$(w,l)$ in~$\R\times(0,T(a_0))$. To conclude, it is
enough to show that~$\Psi$ is a local diffeomorphism
from~$\R\times(0,T(a_0))$ to~$\R\times(0,\infty)$. The Jacobian determinant of the
$\mathcal C^1$ function $\Psi$ is given by
\[
J_\Psi(w,l)
=\Phi_{a_0}'(l)\sqrt{2\Phi_{a_0}(l)}
=f\left(\Phi_{a_0}(l)\right)\, .
\]
From uniqueness in the Cauchy-Lipschitz
theorem, if~$f(\Phi_{a_0}(l))=0$ for some~$l$, then~$f(\Phi_{a_0}(l))=0$
for all~$l\in[0,T(a_0))$, but this contradicts the
assumption~$f(a_0)\neq0$.
As a consequence,~$J_\Psi$ does not vanish on~$\R\times(0,T(a_0))$, so
that~$\Psi$ is a local diffeomorphism, and~$\mu_t$ has a density on~$\R\times(0,\infty)$.
\end{proof}

Without the assumption $x_0=0$, $(X_t)_{t\ge0}$ will
stay away from $0$ for a positive time $\zeta$. In that case
$(A_t)_{t\geq0}$ remains constant on the interval $[0,\zeta]$, and this
results in a more complicated expression for $\mu_t$, as stated in the
following lemma.
\begin{lem}
Assume that $f(a_0)\neq 0$ and~$x_0\neq0$.
Then, $\mu_t$ has the form
\[
\mu_t(x,a)
=m_t(x)\dx x\otimes\delta_{a_0}
+n_t(x,a)\dx x\dx a
+p_t\delta_{(0,0)}
+q_t\delta_{(0,\infty)}\, ,
\]
where $m_t$, $n_t$, are measurable functions
respectively defined on~$\R$ and $\R\times(0,\infty)$.
\end{lem}
\begin{proof}
This relies on the strong Markov property used at time
$\inf\{t>0, X_t=0\}$ together with Lemma \ref{lem:ansatz_loi}.
\end{proof}

The last case to consider is when the process starts from a point where
its diffusion coefficients does not change. In that
case,~$(X_t)_{t\geq0}$ exists for all positive times, and behaves as a
Brownian motion multiplied by some constant.
\begin{lem}
Assume that $f(a_0)=0$.
Then, $\mu_t$ is given by
\[
\mu_t(x,a)
=\gamma_{2ta_0}^{x_0}(x)\dx x\otimes\delta_{a_0}\, ,
\]
where~$\gamma_{\sigma^2}^{x_0}$ denotes the Gaussian distribution with
mean~$x_0$ and variance~$\sigma^2$.
\end{lem}
\begin{proof}
In that case~$(X_t,A_t)_{t\geq0}$ is given by~$(X_t,A_t) = (x_0+\sqrt{2a_0}W_t,a_0)$.
\end{proof}


\subsection{Basic facts about weak solution to~\eqref{eq:EDP2}}

This is a linear equation on $u=u(t,x,a)$ defined on $t \ge 0$, $x \in R$, $a \ge 0$.  We begin with a proper definition of weak solutions, adapted to our context. We recall that $f$ is assumed to be locally Lipschitz continuous from $(0,\infty)$ to $\R$.

  
\begin{defi}\label{def:weak}
We say that $u$ is a weak solution to~\eqref{eq:EDP2}  on $(0,T)$ if it satisfies:
\begin{equation}
u\in L^\infty(0,T;L^1(\R\times \R_+))\, , \quad \partial_x u \in L^1((0,T)\times \R\times \R_+)  \, , \label{eq:flux L1}
\end{equation}
and $u(t,x,a)$ is a solution to~\eqref{eq:EDP2} in the sense of distributions in $\mathcal D'(\R\times \R_+)$: for any test function $\varphi$ in $\mathcal D( \R\times \R_+)$ and a.e. $t\in (0,T)$,
\begin{eqnarray*}
& &\iint u(t,x,a)\varphi(x,a)\dd  x \dd a = \iint u_0(x,a)\varphi(x,a)\dd  x \dd a \\
& &- \int_0^t \iint a\partial _{x} u(s,x,a) \partial_x \varphi(x,a)\dd  x \dd a \dd s + \int_0^t \int f(a)u(s,0,a) \partial_a \varphi(0,a)\dd a \dd s \, .
\end{eqnarray*}
\end{defi}

Since $\partial_x u(t,x,a)$ belongs to $ L^1((0,T)\times \R\times \R_+)$, the solution is well-defined in the distributional sense under assumption~\eqref{eq:flux L1}. In fact we can write
$\int_0^T u(t,0,a) \dd t   = - \int_0^T\int_{x>0} \partial_x u(t,x,a)\dd x\dd t$.

Weak solutions in the sense of Definition~\ref{def:weak} are mass-preserving: 
\[ M =\iint u_0(x,a)\dd x \dd a = \iint u(t,x,a)\dd x \dd a\, .\] 

Let us first prove that non-negativity is preserved.
\begin{lem}
Assume that $u$ is a weak solution to~\eqref{eq:EDP2} such that $\partial^2_{xx}u\in L^{1}$. If $|u_0| = u_0 $ almost everywhere (initial data
non-negative). Then $|u(t,\cdot)|=u(t, \cdot)$ almost everywhere for later times.
\end{lem}
\begin{proof}
Observe that if $u$ is solution in $L^1$ then $|u|$ is subsolution in
$L^1$ since $\mbox{sgn}(u) \partial_{xx}^2 u \le \partial_{xx}^2 |u|$
and
$\mbox{sgn}(u) \partial_a ( f(a) \, u ) \delta_{x=0}
= \partial_a ( f(a) \, |u| ) \delta_{x=0}$. Hence $|u|-u$ is a
subsolution, and 
\begin{align*}
  \dt \iint \left( |u| - u \right) \dd x \dd a \le 0. 
\end{align*}
\end{proof}


Let us next prove that in the case $f(a) \le 0$
(deceleration) and $u_0 \ge 0$, the compact support in $a$ is
preserved along time. 
\begin{lem}\label{lem:support_compact}
Assume $u$ is a weak solution to~\eqref{eq:EDP2} with~$f(a)\leq0$. Assume in addition that $\mbox{supp}(u_0) \subset \R \times [0,a_0]$ for
some $a_0>0$. Then $\mbox{supp}(u) \subset \R \times [0,a_0]$ up to the existence time. 
\end{lem}
\begin{proof}
Consider any non-negative non-decreasing function
$\varphi=\varphi(a)$ smooth on $\R_+$ with support included in
$(a_0,+\infty)$. Then
\begin{align*}
  \dt \iint u(t,x,a) \varphi(a) \dd x \dd a = \int_a f(a) u(0,a)
  \varphi'(a) \dd a \le 0 \, ,
\end{align*}
which proves that $u \varphi =0$ for later times. Varying the
$\varphi$ as defined above we conclude that $u=0$ on $\R \times
[a_0,+\infty)$ for later times.
\end{proof}

Similarly, in the case $f(a) \ge 0$
(acceleration) and $u_0 \ge 0$, we can prove that if the support of
$u_0$ is included in $\R \times [a_0,+\infty)$ for some $a_0>0$, then
the same fact is true for all~$t>0$.

\subsection{The different cases for the law of change in the particular case where $f(a)=\pm a ^\gamma$}
\label{sec:different-cases}

Following the probabilistic study performed in Section \ref{sec:particular_case} we consider the three following cases:
\begin{itemize}
\item[(a)] acceleration at $x=0$: $f(a) = a^\gamma$ with $\gamma \ge 0$,
\item[(b)] subcritical deceleration at $x=0$: $f(a) = -a^\gamma$ with
  $\gamma \geq 3/2$,
\item[(c)] supercritical deceleration at $x=0$: $f(a) = -a^\gamma$
  with $\gamma \in [0,3/2)$.
\end{itemize}
In this part we will prove the following result: 
\begin{theo}\label{th:edp}
Assume that the initial datum $u_0$ belongs to $L^p$, $p\geq1$, then 
\begin{itemize}
\item[(a)] in the acceleration case, there exists a unique weak solution to \eqref{eq:EDP2} that satisfies for all $T>0$, $\sup_{t \in (0,T)}\iint |u(t,x,a)|^p \dd x \dd a<+\infty$,
\item[(b)] in the subcritical deceleration case, for $p=2$, there exists a unique weak solution to \eqref{eq:EDP2} that satisfies for all $T>0$, \\ $\sup_{t \in (0,T)}\iint |u(t,x,a)|^2 \dd x \dd a<+\infty$,
\item[(c)] in the supercritical deceleration case, any weak solution of \eqref{eq:EDP2} blows-up in finite time. 
\end{itemize}
\end{theo}

\begin{proof}[Proof of Theorem \ref{th:edp}]

In the first two cases a) and b), we prove the propagation of $L^p$ bounds, which is the crucial a priori estimate. To prove that solutions blow-up in finite time in the supercritical case c), we show that for an appropriate value of $M$, the momentum $\int a^M u \dd a $ becomes infinite in finite time. 

\subsubsection*{Case (a): global existence and uniqueness}
\label{sec:case-a}
 
Let $p \in [1,+\infty)$, we have the following a priori estimate: 
\begin{align*}
  \dt \iint |u(t,x,a)|^p \dd x \dd a \le& - p(p-1) \iint a \left| \partial_x u(t,x,a)
  \right|^2 |u(t,x,a)|^{p-2} \dd x \dd a \\ & \qquad \qquad -
  \gamma (p-1) \int a^{\gamma-1} |u(t,0,a)|^p \dd a  \le 0\, ,
\end{align*}
which proves that $L^p$ norms remains finite for all times provided
they are finite initially (no finite time appearance of a
singularity). By applying the same argument to the modulus of the
difference of two solutions one proves similarly uniqueness in $L^p$. 

\subsubsection*{Case (b): global existence and uniqueness}
\label{sec:case-b}

In this case the a priori estimate writes for $p=2$:
\begin{align*}
  \dt \iint u^2(t,x,a) \dd x \dd a &\le  - 2\iint a |\partial_x u(t,x,a)|^2\dd x \dd a + \gamma \int
  a^{\gamma-1} u^2(t,0,a) \dd a \\
  & \qquad - \limsup_{a\to 0} a^\gamma u^2(t,0,a) \\
  &\le   - 2\iint a |\partial_x u(t,x,a)|^2\dd x \dd a + \gamma \int
  a^{\gamma-1} u^2(t,0,a) \dd a
\end{align*}
and we control (with $I_\var := [-\var,\var]$)
\begin{align*}
  |u(t,0,a)| &\le \left| u(t,0,a) - \frac1{|I_\var|} \int_{x \in I_\var}
             u(t,x,a) \right| 
             + \left| \frac1{|I_\var|} \int_{x \in I_\var}
             u(t,x,a) \dd x \right| \\
           & \le \left|  \frac1{|I_\var|} \int_{x \in I_\var}
             \Big( u(t,0,a) - u(t,x,a) \Big) \dd x \right| 
             + \frac{1}{\sqrt{2\var}} \| u(t,\cdot,a) \|_{L^2_x(\R)} \\
           & \le \left|  \frac1{2 \var} \int_{-\var} ^\var \int_0 ^x 
             \partial_y u(t,y,a) \dd y \dd x \right|
             + \frac{1}{\sqrt{2\var}} \| u(t,\cdot,a) \|_{L^2_x(\R)} \\
           & \le \left|  \frac1{2 \var} \int_{-\var} ^\var  
             \partial_y u(t,y,a) |\var-y|\dd y \right|
             + \frac{1}{\sqrt{2\var}} \| u(t,\cdot,a) \|_{L^2_x(\R)} \\
           &\le \frac{\var^{1/2}}{2 \sqrt 3} \left\| \partial_x
             u(t,\cdot,a) \right\|_{L^2_x(\R)} 
             + \frac{1}{\sqrt{2\var}} \| u(t,\cdot,a) \|_{L^2_x(\R)}
\end{align*}
and conclude that 
\begin{align*}
  \int  a^{\gamma-1} u^2(t,0,a) \dd a &\le \frac{1}{6}
  \left\| \var^{1/2} a^{(\gamma-1)/2} \partial_x u(t,\cdot,\cdot) \right\|^2_{L^2_{x,a}} \\
& \quad   + \left\| \var^{-1/2} a^{(\gamma-1)/2} u(t,\cdot,\cdot) \right\|^2_{L^2_{x,a}}.
\end{align*}
We choose $\var$ depending on $a$ as $\var = \eta a^{2(\gamma-1)}$
with $\eta$ small to be fixed, and deduce 
\begin{align*}
  \int a^{\gamma-1} u^2(t,0,a) \dd a \le \frac{\eta}{6}
  \left\| a^{(\gamma-1)} \partial_x u(t,\cdot,\cdot) \right\|^2_{L^2_{x,a}} 
  + \frac{1}{\eta} \left\| u(t,\cdot,\cdot) \right\|^2_{L^2_{x,a}}.
\end{align*}
Finally we use that for $\gamma \ge 3/2$ we have $2(\gamma-1) \ge 1$
and therefore on $[0,a_0]$ (remember that the support condition on $a$
is propagated according to Lemma~\ref{lem:support_compact}) we have $a^{2(\gamma-1)} \le C a$. Plugging above we
get 
\begin{align*}
  \int a^{\gamma-1} u^2(t,0,a) \dd a \le \frac{C \eta}{6}
  \left\| a^{1/2} \partial_x u(t,\cdot,\cdot) \right\|^2_{L^2_{x,a}} 
  + \frac{1}{\eta} \left\| u(t,\cdot,\cdot) \right\|^2_{L^2_{x,a}}.
\end{align*}
and 
\begin{equation*}
  \dt \iint u^2 \dd x \dd a \le \left( \frac{C \eta \gamma}{6} - 1
  \right) \iint a |\partial_x u|^2 \dd x \dd a + \frac{\gamma}{\eta} \left\| u \right\|^2_{L^2_{x,a}}.
\end{equation*}
By choosing $\eta < 6/(C \gamma)$, this proves that the $L^2$ norm exists
for all times if it is finite initially.

\subsubsection*{Case (c): blow-up}\label{sec:BU}

First easy step is to compute the evolution for $v(t,a) :=
u(t,0,a)$. We Fourier transform equation \eqref{eq:EDP2} in $x$: 
\begin{align*}
  \partial_t \hat u(t,\xi,a) &= - a |\xi|^2 \hat u(t,\xi,a) + \partial_a
  \left( a^\gamma u(t,0,a) \right) \\
  &= - a |\xi|^2 \hat u(t,\xi,a) + \partial_a
  \left( \frac{a^\gamma}{2\pi} \int_\R \hat u(t,\eta,a) \dd \eta \right)\, ,
\end{align*}
and use Duhamel principle where the last term in the right-hand side is treated as
a source term: 
\begin{equation}\label{eq:fourier}
  \hat u (t,\xi,a) = e^{-t a |\xi|^2} \hat u(0,\xi,a) + \int_0 ^t
  \frac{e^{-(t-s)a |\xi|^2}}{2\pi} \partial_a \left( a^\gamma \int_\R \hat
  u(s,\eta,a) \dd \eta \right) \dd s\, .
\end{equation}
Let 
\[v(t,a) := \frac{1}{2\pi} \int \hat u(t,\eta,a) \dd \eta = u(t,0,a)\, ,\] 
from \eqref{eq:fourier} we deduce the identity:
\begin{equation}\label{eq:def_v}
  v(t,a) = w(t,a) + C_1 \int_0 ^t \frac{1}{\sqrt{a(t-s)}}
  \partial_a \left( a^\gamma v(s,a) \right) \dd s\, , 
\end{equation}
where $C_1>0$ is an explicit constant and $w$ is defined by
\begin{align*}
 w(t,a) := \frac{1}{2\pi}\int_\R e^{-t a |\xi|^2} \hat u(0,\xi,a) \dd \xi\, .
\end{align*}

Then, since $u(s,0,a) \ge 0$, the key remark is that $v(s,a) \ge 0$. Moreover, since the second term in the right-hand side of \eqref{eq:def_v} has zero integral against $a^{1/2}$, one has
\begin{equation*}
\int_{\R_+} a^{1/2} v(t,a) \dd a = \int_{\R_+} a^{1/2} w(t,a) \dd a \le C_2\, .
\end{equation*}
We have therefore a moment bound to start with:
$\int a^{1/2} v(t,a) \dd a$ remains bounded for all times. This
already shows that singularity can only form at $(0,0)$. Hence, we are here
looking only at the value at $x=0$, since for $x \not = 0$ the evolution
is a diffusion in $x$ which does not create singularities. Therefore,
choosing $M \in (0,1/2)$ so that $0<M+\gamma-1/2<M$ (which is
possible since $\gamma <3/2$) we have
\begin{align*}
\int_{\R_+} a^M v(t,a) \dd a & = \int_{\R_+} a^M w(t,a) \dd a + C_1 \int_0 ^t
  \int_{\R_+} \frac{a^{M-1/2}}{\sqrt{(t-s)}} \partial_a \left( a^\gamma
  v(s,a) \right) \dd a \dd s \\ 
  & = \int_{\R_+} a^M w(t,a) \dd a \\
  & \quad - C_1 \left( M - \frac12 \right) \int_0 ^t
  \int_{\R_+} \frac{a^{M+\gamma-3/2}}{\sqrt{(t-s)}} 
  v(s,a) \dd a \dd s\, .
\end{align*}
Note that since $M$ satisfies  $M-1/2+\gamma >0$, in the previous computation the boundary term due to the integration by parts vanishes:
  \[ \int_0 ^t \left[ \frac{a^{M+\gamma-1/2}}{\sqrt{(t-s)}} v(s,a)
  \right]_{a=0} ^{a=+\infty} \dd s =0\, .\]
Furthermore, there exists some positive constants $C_3$ and $\eta$, with $M-\eta \in (0,1/2)$, such that $-(M-1/2) a^{M+\gamma-3/2} \ge C_3 a^{M-\eta}$ on the compact support $[0,a_0]$. And we deduce that
\begin{align*}
\int_{\R_+} a^M v(t,a) \dd a &\ge \int a^M w(t,a) \dd a \\
&+\,  C_4 (1+T)^{-1/2} \int_0 ^t
  \int_{\R_+} a^{M-\eta} v(s,a) \dd a \dd s
\end{align*}
for a constant $C_4$. We have by interpolation (using
$M-\eta < M <1/2$)
\begin{align*}
  \int_{\R_+} a^M v(s,a) \dd a \le \left( \int_{\R_+} a^{1/2} v(s,a) \dd a
  \right)^\theta \left( \int_{\R_+} a^{M-\eta} v(s,a) \dd a \right)^{1-\theta}
\end{align*}
for some $\theta \in (0,1)$. Finally, using the bound on the $1/2$-moment,
we obtain that 
\begin{align*}
\int_{\R_+} a^M v(t,a) \dd a &\ge \int_{\R_+} a^M w(t,a) \dd a \\
&+\  C_5 (1+T)^{-1/2} \int_0 ^t
  \left( \int_{\R_+} a^{M} v(s,a) \dd a \right)^{1/(1-\theta)} \dd s
\end{align*}
for some constant $C_5>0$. This means that
$Y(t) := \int a^M v(t,a) \dd a$ satisfies on $[0,T]$:
\begin{align*}
  Y(0) >0 \quad \mbox{ and } \quad Y'(t) \ge C_5 (1+T)^{-1/2}
  Y(t)^{1+\alpha}\, , 
\end{align*}
where $1+ \alpha := 1/(1-\theta) >1$.
This implies on $[0,T]$ that 
\begin{align*}
  Y(t) \ge \left[ \frac{1}{Y(0)^\alpha - \alpha C_5(1+T)^{-1/2}t} \right]^{1/\alpha}
\end{align*}
At $t=T/2$ we have $\alpha C_5(1+T)^{-1/2}t = \alpha
C_5(1+T)^{-1/2}T/2$ goes to infinity as $T$ goes to infinity, hence by
taking $T$ large enough, we find that $Y(t)$ becomes infinite in
finite time. 

Using next that 
\[\int _{\R_+} a^M v(t,a) \dx a \le \int_{\R_+} v(t,a) \dx a + \int _{a\ge 1} a^{1/2} v(t,a) \dx a \, ,
\]
together with $\int _{\R_+} a^{1/2} v(t,a) \dx a <\infty$, we deduce that the $L^1$ norm can not stay finite for all time. The same reasoning  holds true for all $L^p$ with Holder inequality.
\end{proof}


\subsection{Boundary value problem associated to \eqref{eq:EDP2}}

In \cite{benichou-meunier-redner-voituriez-12}, an explicit solution was obtained by using Laplace transform on the boundary value problem associated to \eqref{eq:EDP2} for a particular initial condition $n_{0}(x,a)=\delta_{(0,a_0)}$.  In this section we recall such a result and we compare it with the result of Theorem \ref{th:edp}. 

First, we transform equation \eqref{eq:EDP2}, which includes a singular coefficient, namely $\delta_{x=0}$, into a boundary value problem with regular coefficients. In this part, we will use both notations $n(t,x,a),p(t)$ and $n_t(x,a),p_t$.
\begin{prop}\label{prop:edp}
Assume that $f(0)=0$, $a_0>0$ and that $f(a_0)\neq 0$. Assume in addition that $n_{0}(x,a)=\delta_{(0,a_0)} $, $p_0=0$ and $q_0=0$, then $(n_t,p_t,q_t)$ given in
Lemma \ref{lem:ansatz_loi} satisfies the following problem in the
classical sense:
\begin{equation}\label{eq:pb_CI_diff}
\begin{cases}
\partial_t n(t,x,a)
=a\ \partial_x^2n(t,x,a)
\ \textrm{ for } (t,x,a)\in \R^*\times \R^* \times \R_+^*\,, \\
a \left(\partial_x n(t,0^+,a) - \partial_x n(t,0^-,a)\right)
= \partial_a\left(f(a)n(t,0,a)\right)\, ,
 a\in \R_+^*\, .
\end{cases}
\end{equation}
Furthermore, one has 
\begin{itemize}
\item[] $ \lim_{a\to 0} \left(f(a)n(t,0,a) \right)
= -p'(t)$, with $p'(t)=0$ if $f(a) > 0$,
\item[] $\lim_{a\to +\infty} \left(f(a)n(t,0,a) \right)
=q'(t)$,  with $q'(t)=0$ if $f(a) < 0$.
\end{itemize}
\end{prop}
\begin{proof}
We use Lemma \ref{lem:ansatz_loi} with $n_0(x,a)=
\delta_{(0,a_0)}$ and we assume that $f(0)=0$. From the space
symmetry of the process with respect to the origin, we first notice that 
$n_t(x,a)=n_t(-x,a)$ for all $(t,x,a)\in \R_+^* \times \R\times
\R_+$. Next, as it is classical in such situations, in order to obtain
\eqref{eq:pb_CI_diff}, we multiply equation \eqref{eq:EDP2} by particular
test functions $\phi$ and we integrate by parts: respectively $\phi\in \mathcal{C}_c^\infty \left(\R^* _\pm\times \R_+^*\right)$ and $\phi \in \mathcal{C}_c^\infty \left(\R\times \R_+^*\right) $.
Using the smoothing effect of the heat equation, we can see that Equation \eqref{eq:pb_CI_diff} is satisfied in the classical sense.
\end{proof}

Let us now introduce some notations that will be useful for the computation of the explicit solution to problem \eqref{eq:pb_CI_diff}:  
$H$ will denote the Heaviside function, $H(a-a_0)=0$ if
$a<a_0$, $H(a-a_0)=1$ if $a>a_0$ and $Z$ will be the function defined by
\begin{equation}\label{eq:def_sol_explicite}
Z(x,a)
=\frac{|x|}{\sqrt{a}}+2\int_{a_0}^{a} \frac{\sqrt{a'}}{f(a') } \dx a'.
\end{equation}
\begin{prop}\label{prop:explicit}
The boundary value problem \eqref{eq:pb_CI_diff} with $n_{0}(x,a)=\delta_{(0,a_0)} $ and $p_0=0$ admits the following explicit solution $(n_t,p_t)$:
\begin{enumerate}
\item Local deceleration ($f(a) < 0$):
\begin{eqnarray}\label{eq:sol_explicite_decelere}
n_t(x,a)\!=\! H(a_0-a)\, \frac{Z(x,a)}{|f(a)|\sqrt{4\pi t^3}} \, e^{-\frac{Z(x,a)^2}{4t}} \\
\ \textrm{ and } \
p_t
=\int_0^t{\rm erfc}\Big(\frac{1}{\sqrt{s}}\int_{0}^{a_0}\frac{\sqrt{a'}}{|f(a')|}
\,\dx a'\Big) \dx s\, ,\nonumber
\end{eqnarray}
\item Local acceleration ($f(a) >0$):
\begin{equation}\label{eq:sol_explicite_accelere}
n_t(x,a)=H(a-a_0)\frac{Z}{f(a)\sqrt{4\pi t^3}}\,\, e^{-\frac{Z(x,a)^2}{4t}}
\quad \textrm{ and } \quad
p_t=0\, .
\end{equation}
\end{enumerate}
\end{prop}
\begin{proof}
Taking the Laplace transform in $t$, the Fourier transform
in $x$ of \eqref{eq:pb_CI_diff} with
$n_{0}(x,a)=\delta_{(0,a_0)} $ and using the
equality
\[
\mathcal{L}^t(n) (0,a,\lambda)
=\frac{1}{2\pi}
\int_{\R} \tilde{n}_\lambda (\xi,a,\lambda) e^{i\xi \times 0}\dx \xi \, ,
\]
we can compute
\begin{equation}\label{eq:dirac_lap}
\tilde{n}_\lambda (\xi,a) 
= \frac{1}{\lambda + a |\xi|^2}
\partial_a\left( f(a) \left(
    \frac{1}{2\pi}\int_{\R} \tilde{n}_\lambda (\xi',a)
    \dx \xi' \right)\right) 
+ \frac{\delta _{a=a_0}}{\lambda + a  |\xi|^2}\, ,
\end{equation}
where we have denoted by $\tilde{n}_\lambda (\xi,a )= \int_{\R}\mathcal{L}^t(n)(x,a,\lambda) e^{-i\xi x} \dx x$ the Fourier transform in
$x$ of the Laplace transform in $t$ of $n_t$, which is denoted by $\mathcal{L}^t(n) (x,a,\lambda)  = \int_{\R_+}n_t(x,a) e^{-\lambda t} \dx t$.

Consequently after integration we obtain
\begin{align}
\int_\R\tilde n_\lambda(\xi,a)\dx\xi
&=\left[\partial_a
  \left(f(a)\left(\frac1{2\pi}\int_\R\tilde n_\lambda(\xi',a)\dx\xi'\right)\right)
  + \delta_{a=a_0}\right]
\int _{\R} \frac1{\lambda + a |\xi|^2} \dx \xi \nonumber \\
&= \frac{\pi }{\sqrt{\lambda a}}
\left[
  \partial_a \left( f(a) \left( \frac{1}{2\pi}\int_{\R}\tilde{n}_\lambda (\xi,a) \dx \xi \right)\right)
  + \delta_{a=a_0}
\right] \, .
\label{eq:eq_fermee}
\end{align}
With the previous expression,  $\int_{\R} \tilde{n}_\lambda (\xi,a) \dd \xi $ can be 
computed, hence the expression of $\tilde{n}_\lambda $ can be
deduced by using \eqref{eq:dirac_lap}. Finally, by inverting first the Fourier transform
and then the Laplace tranform, we can compute $n_t$. From
\eqref{eq:dirac_lap} and \eqref{eq:eq_fermee} it follows that
\begin{eqnarray*}
\tilde n_\lambda (\xi,a)
&=&\frac1{\pi } \frac{\sqrt{\lambda a}}{\lambda +a \xi ^2}
\int_{\R}\tilde{n}_\lambda (\xi,a) \dx \xi\\
&=&C(a,a_0, \lambda)\frac{f(a_0)}{2\pi f(a)}
\frac{\sqrt{\lambda a}}{\lambda +a \xi ^2}
e^{2\sqrt{\lambda} \int_{a_0}^a \frac{\sqrt{a'}}{f(a')} \dx a'}\, ,
\end{eqnarray*}
where
\begin{equation*}
C(a,a_0,\lambda) =
\begin{cases}
C(\lambda) H(a-a_0) \quad \textrm{ if } f>0\, ,\\
C(\lambda) H(a_0-a) \quad \textrm{ if } f<0 \, .
\end{cases}
\end{equation*}
$C(\lambda)$ is determined by the jump of
$\int_{\R}\tilde{n}_\lambda (\xi,a) \dx \xi$ at $a=a_0$:
\begin{equation}\label{eq:saut}
\int _{\R}\tilde{n}_\lambda (\xi,a_0^+)\dx \xi
-\int _{\R}\tilde{n}_\lambda (\xi,a_0^-)\dx \xi
= \frac{1}{f(a_0)}\, ,
\end{equation}
from which we deduce that $C(\lambda) = 1/|f(a_0)|$, hence
\begin{equation*}
\tilde{n}_\lambda (\xi,a) 
= C(a,a_0)\frac{1}{2\pi |f(a)|} \frac{\sqrt{\lambda a}}{\lambda +a \xi^2}
e^{2\sqrt{\lambda} \int_{a_0}^a \frac{\sqrt{a'}}{f(a')} \dx a'} \, .
\end{equation*}

Next, using the Fourier inverse transform, it yields that
\begin{eqnarray*}
 \mathcal{L}^t(n) (x,a,\lambda)
  &=& C(a,a_0)\frac{1}{ |f(a)|}
  e^{-\left| \sqrt{\frac{\lambda}{a}}x\right|+2\sqrt{\lambda} \int_{a_0}^a \frac{\sqrt{a'}}{f(a')} \dx a'}\\
 & =& C(a,a_0)\frac{1}{ |f(a)|}
  e^{- \sqrt{\lambda} Z(x,a)} \, ,
\end{eqnarray*}
where $Z$ is given by \eqref{eq:def_sol_explicite}.

Laplace inverting this latter expression, we obtain the joint
distribution given by \eqref{eq:sol_explicite_decelere} if $f<0$ and by
\eqref{eq:sol_explicite_accelere} if $f>0$. The expression of $p_t$ is straigtforward.
\end{proof}

\begin{rem}
Unsurprisingly the blow-up character of the solution to \eqref{eq:EDP2} given in Theorem \ref{th:edp} can also be observed on the explicit solution given in Proposition \ref{prop:explicit}. Indeed the symptom of blow-up in a $L^p$ framework corresponds to $p_t\neq 0$. In the case where $f(a)=-a ^\gamma$, we see that $\int_0^a a'^{1/2-\gamma}\dx a'<\infty$ if $\gamma \in (0,3/2)$. Since Laplace transform inversion requires a specific initial condition the result given in Theorem \ref{th:edp} is more general. 
\end{rem}

\bibliographystyle{plain}
\def\cprime{$'$} \def\lfhook#1{\setbox0=\hbox{#1}{\ooalign{\hidewidth
  \lower1.5ex\hbox{'}\hidewidth\crcr\unhbox0}}} \def\cprime{$'$}
  \def\cprime{$'$} \def\cprime{$'$} \def\cprime{$'$} \def\cprime{$'$}

\end{document}